\documentclass[11pt,a4paper]{amsart}
\usepackage{url}
\usepackage{amscd}
\usepackage{xcolor}

\usepackage[backref,
pagebackref,%
bookmarks=true,%
colorlinks=true,
linkcolor=blue
]%
{hyperref}

\newtheorem{theorem}{Theorem}
\theoremstyle{plain}

\newtheorem{theo}{Theorem}[section]
\newtheorem{lemma}[theorem]{Lemma}
\newtheorem{prop}[theorem]{Proposition}

\theoremstyle{definition}

\theoremstyle{remark}
\newtheorem{remark}[theorem]{Remark}

\numberwithin{equation}{section}

\DeclareMathOperator{\md}{mod}

\begin{document}

\title[Ergodicity on 2-adic spheres]{Ergodicity criteria for non-expanding transformations of 2-adic spheres}

\author{Vladimir Anashin}
\author{Andrei Khrennikov} 
\author{Ekaterina Yurova}
\thanks{The first of the authors is supported in parts by Russian Foundation for Basic
Research grant 12-01-00680-a and by  Chinese Academy of Sciences visiting professorship for senior international
scientists grant No  2009G2-11.
The second and the third of the authors are supported by the joint grant of Swedish and South-African Research Councils ``Non-Archimedean analysis: from fundamentals 
to applications.''}

%
%
%

\begin{abstract}
In the  paper, we obtain necessary and sufficient conditions for ergodicity
(with respect to the normalized Haar measure) of discrete dynamical systems $\langle f;\mathbf S_{2^{-r}}(a)\rangle$
on 2-adic spheres $\mathbf S_{2^{-r}}(a)$ of radius $2^{-r}$, $r\ge 1$, centered
at some point $a$ from the ultrametric space of 2-adic integers $\mathbb Z_2$.  The map $f\colon\mathbb Z_2\to\mathbb Z_2$ is assumed to be non-expanding
and measure-preserving; that is, $f$ satisfies
a Lipschitz condition with a constant 1 with respect to the 2-adic metric,
and $f$ preserves a natural probability measure on $\mathbb Z_2$, the Haar measure $\mu_2$ on $\mathbb Z_2$ which is normalized so that $\mu_2(\mathbb
Z_2)=1$.
\end{abstract}

\maketitle
 
 \section{Introduction}
 
 Algebraic and arithmetic dynamics are actively developed fields of general theory of dynamical systems, see \cite{Vivaldi} for extended 
bibliography, also monographs \cite{Silverman,ANKH}. Theory of dynamical systems in $p$-adic fields ${\mathbb Q}_p,$ where $p\geq 2$ is a prime number,  is  an important part of 
algebraic and arithmetic dynamics, see, e.g., \cite{Albeverio1,Albeverio2,Vivaldi1,Vivaldi2}; also 
\cite{Vivaldi,Khrennikov and Nilsson}, and \cite{ANKH} for further references. As in general theory of dynamical systems, problems  of measure preservation
and ergodicity play fundamental role in
the theory of $p$-adic dynamical systems, see e.g. \cite{ANKH,Erg,Arrowsmith1,Arrowsmith2,Gundlach1,Gundlach2,Khrennikov and Nilsson,Parry,Benedetto1,Benedetto2,Benedetto3,Benedetto4,Chabert,Fan3,Rivera1,Rivera2,Rivera3,Rivera4,Lindhal,De Smedt1}. 

The case of  non-expanding dynamics (the ones that satisfy a Lipschitz
condition with a constant 1, a 1-Lipschitz for short) 
on the ring $\mathbb Z_p$ of $p$-adic integers  is sufficiently well studied
\cite{Unif0,Unif,Fan1,Fan2},
see also \cite{ANKH} and references therein.
However, it is not so much known about the dynamics in domains different from
$\mathbb Z_p$ 
although the later dynamics can be useful 
in applications to computer science (e.g. in computer simulations,
numerical methods like Monte-Carlo, cryptography) and to mathematical physics,
see \cite{ANKH}, \cite{me-NATO} and \cite{Khrennikov and Nilsson}. Dynamical systems on $p$-adic spheres are an interesting and nontrivial
example of the dynamics.  The first result in this direction, namely, the ergodicity criterion for monomial dynamical systems on $p$-adic spheres, was obtained in \cite{Gundlach1,Gundlach2}. It deserves a note that although
these dynamical systems are a $p$-adic counterpart of a classical dynamical
systems, circle rotations,
in the $p$-adic case the dynamics exhibit
quite another behavior than the classical one.
Later the case of monomial dynamical systems on $p$-adic spheres was significantly
extended: In \cite{Erg}, ergodicity criteria for locally analytic dynamical systems on $p$-adic spheres were obtained, for arbitrary prime $p$. 
 
In the current paper, we consider 
an essentially wider class of dynamics than in \cite{Erg}, namely, the class of all 1-Lipschitz dynamical systems; however, on 
2-adic spheres only: We find necessary and sufficient conditions for ergodicity of these dynamical systems, see further Theorem \ref{thm:ergsph}. Then with
the use of the criterion we find necessary and sufficient conditions for ergodicity of perturbed monomial dynamical systems on 2-adic spheres around
1 in the case when perturbations are 1-Lipschitz and 2-adically small
(Theorem \ref{thm:pert}). Earlier similar results were known only under
additional restriction (however,
for arbitrary prime $p$) that
perturbations are smooth, cf. \cite{Erg} and \cite[Section 4.7]{ANKH}. In this connection it should be noticed that transition of  results of the present paper for arbitrary prime $p$ seems to be a non-trivial task:
It is well known that 
in $p$-adic analysis cases of even and odd primes differ essentially. 

Our basic technique is van der Put series.
 The van der Put series were primarily known only as a tool to find antiderivatives (see  \cite{vdp,Mahler,Shihov}); recently  by using the series the authors in \cite{DAN,AKY,Ya}
developed a new technique  to determine whether a 1-Lipschitz transformation is
measure-preserving and/or ergodic on $\mathbb Z_p$. Our approach
seems to be fruitful: 
The analog of the techniques was successfully applied to determine ergodicity
of 1-Lipschitz transformations on another complete non-Archimedean ring, the ring
$\mathbb F_2[[X]]$
of formal power series over a two-element field $\mathbb F_2$, see recent
paper \cite{LinTaoYang}.
 
We remark that as the mappings under consideration are in general not differentiable, it is impossible to apply the technique based on expansion into power series to  the case under consideration as the said technique can be
used for analytical and smooth dynamical systems only, see e.g. \cite{Erg}. The van der Put basis is much better adapted 
to studies of non-smooth dynamics: A special  collection of step-like functions,
characteristic functions of balls, constitutes the basis.
The 
van der Put basis reflects the ultrameric (non-Archimedean) structure of
$p$-adic numbers, \cite{Mahler,Shihov}. We note that in the $p$-adic case the linear space consisting of linear combinations of step-like functions is a dense subspace of the space of continuous functions, \cite{Shihov}.

The 2-adic spheres are a special case of $p$-adic spheres; as a matter of
fact, 2-adic
spheres are 2-adic balls: Denote the $p$-adic absolute value via $|~|_p$;
then, as  a $p$-adic sphere 
$\mathbf S_{p^{-r}}(a)=\{z\in\mathbb Z_p\colon|z-a|_p=p^{-r}\}$ of radius $p^{-r}$ centered at $a\in\mathbb Z_p$ is a disjoint union of $p-1$ balls
$\mathbf B_{p{-r-1}}(b)=\{z\in\mathbb Z_p\colon|z-b|_p\le p^{-r-1}\}=b+p^{r+1}\mathbb Z_p$
of radii $p^{-r-1}$ cantered at $b\in\{a+p^rs\colon s=1,2,\ldots,p-1\}$,
in the case $p=2$ we get that $\mathbf S_{2^{-r}}(a)=\mathbf B_{2^{-r-1}}(a+2^r)$. Fortunately, 
the problem to determine ergodicity  of a 1-Lipschitz transformation on a ball $\mathbf B_{p^{-k}}(a)=a+p^k\mathbb Z_p\subset\mathbb Z_p$ can be reduced
to the same problem on the whole space $\mathbb Z_p$. 

Indeed, if $f$ is a 1-Lipschitz transformation 
such that $f(a+p^k\mathbb Z_p)\subset a+p^k\mathbb Z_p$, then necessarily $f(a)=a+p^ky$ for
a suitable $y\in\mathbb Z_p$. 
Thus, $f(a+p^kz)=f(a)+p^k\cdot u(z)$ for any $z\in\mathbb Z_p$; so we can relate to $f$ the following 1-Lipschitz transformation on $\mathbb Z_p$:
$$
u\colon z\mapsto u(z)=\frac{1}{p^k}(f(a+p^kz)-a-p^ky); \ \ z\in\mathbb Z_p.
$$
It can be shown that the transformation $f$ is ergodic
on the ball $\mathbf B_{p^{-k}}(a)$  if and only if the transformation  $u$
is ergodic on $\mathbb Z_p$.  To determine ergodicity of a transformation
on the space $\mathbb Z_p$ various techniques may be used, see \cite{ANKH} for details. In
the paper, we exploit a version of the idea described above to reduce the case of ergodicity on  2-adic spheres to
the case of ergodicity on the whole space $\mathbb Z_2$ (cf. further Proposition \ref{prop:ball-space}),
and we use van
der Put series for the latter study since the series turned out
to be the
most effective technique in the case when $p=2$, cf.  \cite{DAN,AKY,LinTaoYang,Ya}.
 
\section{Preliminaries}  
We remind that  $p$-adic absolute
value satisfies {\it strong triangle inequality}:
$$
\vert x+y\vert_p \leq \max \{\vert x\vert_p, \vert y\vert_p\}.
$$
The $p$-adic absolute value induces ($p$-adic) metric on $\mathbb Z_p$ in
a standard way: given $a,b\in\mathbb Z_p$, the $p$-adic distance between
$a$ and $b$ is $\vert a-b\vert_p$.
Absolute values (and also metrics induces by these absolute values) that satisfy  strong triangle inequality are called \emph{non-Archimedean}.
Although the strong triangle inequality
is the only difference of the $p$-adic metric from   real
or complex metrics it results in
dramatic differences in behaviour of $p$-adic dynamical systems compared
to that of  real or complex counterparts.

The space $\mathbb Z_p$ is equipped with a natural probability measure, namely, the Haar measure $\mu_p$ normalized
so that $\mu_p(\mathbb Z_p) = 1$: Balls $\mathbf B_{p^{-r}}(a)$ of non-zero radii constitute
the base of the corresponding $\sigma$-algebra of measurable subsets, 
$\mu_p(\mathbf B_{p^{-r}}(a))=p^{-r}$. The measure $\mu_p$ is a regular Borel measure,
so all continuous transformations $f\colon\mathbb Z_p\rightarrow\mathbb Z_p$ are measurable with respect to
$\mu_p$.
As usual, a measurable mapping $f\colon\mathbb Z_p\rightarrow \mathbb Z_p$
is called \emph{measure-preserving} if 
$\mu_p(f^{-1}(\mathbf S))=\mu(\mathbf S)$ for each measurable subset $\mathbf
S\subset \mathbb Z_p$.
A measure-preserving mapping $f\colon\mathbb Z_p\rightarrow\mathbb Z_p$ is called \emph{ergodic} if $f^{-1}(\mathbf S)=\mathbf
S$ implies either
$\mu_p(\mathbf S)=0$ or $\mu_p(\mathbf S)=1$ (in the paper, speaking of ergodic
mapping we mean that the mappings are also measure-preserving).

Let a transformation $f\colon\mathbb Z_p\rightarrow\mathbb Z_p$ be non-expanding
with respect to the $p$-adic metric; that is, let $f$ be  a 1-Lipschitz with respect to the $p$-adic metric:
 $$\left| f(x)- f(y)\right|_p \leq \left|x-y \right|_p$$
for all $x,y \in \mathbb Z_p$. 
The 1-Lipschitz property may be re-stated in terms
of congruences rather than in term of inequalities, in the following way.

Given $a,b\in\mathbb Z_p$ and $k\in\mathbb N=\{1,2,3,\ldots\}$, the congruence $a\equiv b\pmod{p^k}$ is well defined: the congruence just means that images of $a$ of $b$ under
action of the ring epimorphism $\bmod p^k\colon\mathbb Z_p\rightarrow\mathbb Z/p^k\mathbb Z$ of the ring $\mathbb Z_p$
onto the residue ring $\mathbb Z/p^k\mathbb Z$ modulo $p^k$ coincide. Remind that by the
definition the  epimorphism $\bmod
p^k$ sends a $p$-adic integer that has a \emph{canonic representation} $\sum_{i=0}^\infty\alpha_ip^i$,
$\alpha_i\in\{0,1,\ldots,p-1\}$, $i=0,1,2,\ldots$, to  $\sum_{i=0}^{k-1}\alpha_ip^i\in\mathbb Z/p^k\mathbb Z$.
Note also that we treat if necessary elements from $\mathbb Z/p^k\mathbb Z$ as numbers from $\{0,1,\ldots,p^k-1\}$. 

Now it is obvious that the congruence $a\equiv b\pmod{p^k}$ is equivalent
to the inequality $\vert a-b\vert_p\le p^{-k}$. Therefore the transformation $f\colon\mathbb Z_p\rightarrow\mathbb Z_p$
is 1-Lipschitz if and only if it is \emph{compatible}; that is,
\begin{equation}
\label{1L=comp}
f(a)\equiv f(b)\pmod {p^k} \ \text{once}\ a\equiv b \pmod{p^k}.
\end{equation}
The compatibility property implies that given a 1-Lipschitz transformation $f\colon\mathbb Z_p\rightarrow\mathbb Z_p$,
the \emph{reduced mapping modulo $p^k$} $$f\md p^k\colon z\md p^k\mapsto
f(z)\md p^k$$ is a well defined mapping $f\md p^k\colon\mathbb Z/p^k\mathbb Z\rightarrow\mathbb Z/p^k\mathbb Z_p$ of residue
ring $\mathbb Z/p^k\mathbb Z_p$ into itself: The mapping $f\md p^k$ does not depend on
the choice of representative $z$ in the ball $z+p^k\mathbb Z_p$ (the latter ball is a coset with respect to
the epimorphism $\bmod p^k$); that is, the following diagram commutes: 
$$
\begin{CD}
\mathbb Z_p@>f>>\mathbb Z_p\\
@VV\md p^k V @VV\md p^k V\\
\mathbb Z/p^k\mathbb Z @>f\md p^k>>\mathbb Z/p^k\mathbb Z
\end{CD}
$$
A 1-Lipschitz transformation $f\colon\mathbb Z_p\rightarrow\mathbb Z_p$ is called \emph{bijective modulo $p^k$} if the reduced mapping $f\md p^k$ is a permutation on $\mathbb Z/p^k\mathbb Z$; and $f$
is called \emph{transitive modulo $p^k$} 
if $f\md p^k$ is a permutation that 
is cycle of length $p^k$.  
Main ergodic theorem for 1-Lipschitz transformations on $\mathbb Z_p$  \cite[Theorem 4.23]{ANKH}
yields: 
\begin{theo}[Main ergodic theorem]
\label{thm:erg-main}
A 1-Lipschitz transformation $f\colon\mathbb Z_p\rightarrow\mathbb Z_p$ is measure-preserving if and only if it is bijective modulo $p^k$ for all $k=1,2,3,\ldots$; and
$f$ is ergodic if and only if $f$ is transitive modulo $p^k$ for all $k=1,2,3,\ldots$.
\end{theo}

Now we remind definition and basic properties of  van der Put series  following
\cite{Mahler}. Given a continuous 
$p$-adic function 
$f\colon \mathbb Z_p\rightarrow \mathbb Z_p$ defined on $\mathbb Z_p$ and
valuated in $\mathbb Z_p$, 
there exists 
a unique sequence 
$B_0,B_1,B_2, \ldots $ 
of $p$-adic integers such that 

\begin{equation}
\label{vdp}
f(x)=\sum_{m=0}^{\infty}
B_m \chi(m,x) 
\end{equation}
for all $x \in \mathbb Z_p$, where 
\begin{displaymath}
\chi(m,x)=\left\{ \begin{array}{cl}
1, &\text{if}
\ \left|x-m \right|_p \leq p^{-n} \\
0, & \text{otherwise}
\end{array} \right.
\end{displaymath}
and $n=1$ if $m=0$; $n$ is uniquely defined by the inequality 
$p^{n-1}\leq m \leq p^n-1$ otherwise. The right side series in \eqref{vdp} is called the \emph{van der Put series} of the function $f$. Note that
the sequence $B_0, B_1,\ldots,B_m,\ldots$ of \emph{van der Put coefficients} of
the function $f$ tends $p$-adically to $0$ as $m\to\infty$, and the series
converges uniformly on $\mathbb Z_p$. Vice versa, if a sequence $B_0, B_1,\ldots,B_m,\ldots$
of $p$-adic integers tends $p$-adically to $0$ as $m\to\infty$, then the the series in the right-hand
side of \eqref{vdp} converges uniformly on $\mathbb Z_p$ and thus defines a continuous
function $f\colon \mathbb Z_p\to\mathbb Z_p$.

The number $n$ in the definition of $\chi(m,x)$ has a very natural meaning:
As
\begin{equation}
\label{eq:floor_log}
\left\lfloor  \log_p m \right\rfloor 
=
\left(\text{\emph{the number of digits in a base-}} p \;\text{\emph{expansion for}} \;m\right)-1,
\end{equation}
therefore $n=\left\lfloor  \log_p m \right\rfloor+1$ for all $m\in\mathbb N_0$;
we put  $\left\lfloor  \log_p 0 \right\rfloor=0$ by this reason. 
Recall that $\lfloor\alpha \rfloor$ for a real $\alpha $ denotes the integral part of $\alpha$, that is, the nearest
to $\alpha $ rational integer which does not exceed  $\alpha $.

Coefficients $B_m$ are
related to values of the function $f$ in the following way:
Let 
$m=m_0+ \ldots +m_{n-2} p^{n-2}+m_{n-1} p^{n-1}$ be a base-$p$ expansion
for $m$, i.e., 
 $ m_j\in \left\{0,\ldots ,p-1\right\}$, $j=0,1,\ldots,n-1$ and $m_{n-1}\neq 0$, then
\begin{equation}
\label{eq:vdp-coef}
B_m=
\begin{cases}
f(m)-f(m-m_{n-1} p^{n-1}),\ &\text{if}\ m\geq p;\\
f(m),\ &\text{if otherwise}.
 \end{cases}
\end{equation}
It is worth noticing  also that $\chi (m,x)$ is merely  a \emph{characteristic function of the ball} $\mathbf B_{p^{-\left\lfloor  \log_p m \right\rfloor-1}}(m)=m+p^{\left\lfloor  \log_p m \right\rfloor-1}\mathbb Z_p$
of radius $p^{-\left\lfloor  \log_p m \right\rfloor-1}$ centered at $m\in\mathbb N_0=\{0,1,2,\ldots\}$:
\begin{multline}
\label{eq:chi}
\chi(m,x)=\begin{cases}
1,\ &\text{if}\ x\equiv m \pmod{p^{\left\lfloor  \log_p m \right\rfloor+1}};\\
0,\ &\text{if otherwise}
 \end{cases}
 =\\
\begin{cases}
1,\ &\text{if}\ x\in \mathbf B_{p^{-\left\lfloor  \log_p m \right\rfloor-1}}(m);\\
0,\ &\text{if otherwise}
\end{cases}
\end{multline}
The following theorem that characterizes 1-Lipschitz functions in terms of van
der Put basis was proved in \cite{DAN}:
\begin{theo}
\label{thm:comp}
Let a function $f\colon \mathbb Z_p\rightarrow \mathbb Z_p$ be represented via van der Put series
\eqref{vdp};  then $f$ is compatible \textup{(that is, satisfies the $p$-adic
Lipschitz
condition with a constant 1)} if and only if $|B_m|_p\le
p^{-\left\lfloor \log_p m \right\rfloor}$ for all $m=0,1,2,\ldots$. 
\end{theo}
In other words, $f$ is compatible if and only if it can be represented as
\begin{equation}
\label{eq:vdp-comp}
f(x)=\sum_{m=0}^{\infty}
p^{\left\lfloor \log_p m \right\rfloor} b_m \chi(m,x),
\end{equation}
for suitable $b_m\in \mathbb Z_p$; $m=0,1,2,\ldots$.

To study ergodicity on $p$-adic spheres, the following lemma is useful (further
$f^i$ stands for the $i$-th iterate of $f$):
\begin{lemma}[\protect{\cite[Lemma 4.76]{ANKH}}]
\label{lem:transit}
A 1-Lipschitz transformation $f\colon \mathbb Z_p\rightarrow
\mathbb Z_p$ is ergodic on the sphere $\mathbf S_{p^{-r}}(y) $ if and only if the following
two conditions hold simultaneously:
\begin{enumerate}
\item The mapping $z\mapsto f(z)\md p^{r+1}$ is transitive on the set
$$\mathbf S_{p^{-r}}(y)\md p^{r+1}=\{y+p^rs\colon s=1,2,\ldots,p-1\}\subset\mathbb Z/p^{r+1}\mathbb Z;$$
\item The mapping $z\mapsto f^{p-1}(z)\md p^{r+t+1}$ is transitive on  the set
$$
\mathbf B_{p^{{-(r+1)}}}(y+p^rs)\md p^{r+t+1}=\{y+p^rs+p^{r+1}S\colon S=0,1,2,\ldots,
p^{t}-1\},
$$
for all $t=1,2,\ldots$ and for some \textup{(}equivalently, for all\textup{)} $s\in\{1,2,\ldots, p-1\}$.
\end{enumerate}
Condition 2 holds if and only if $f^{p-1}$ is an ergodic transformation on the ball $\mathbf B_{p^{{-(r+1)}}}(y+p^rs)=y+p^rs+p^{r+1}\mathbb Z_p$ of radius $p^{-r-1}$
centered at $y+p^rs$,
for some \textup{(}equivalently, for all\textup{)} $s\in\{1,2,\ldots, p-1\}$.
\end{lemma}

The ergodic 1-Lipschitz transformations of $\mathbb Z_2$ are completely characterized
by the following theorem, see \cite{DAN}:
\begin{theo} 
\label{thm:ergnew}
A 1-Lipschitz transformation $f\colon \mathbb Z_2\rightarrow \mathbb Z_2$ is ergodic if and only if
it can be represented as
\[
f(x)= b_0\chi(0,x)+ b_1\chi(1,x)+  
\sum_{m=2}^{\infty}
2^{\left\lfloor \log_2 m \right\rfloor} b_m \chi(m,x)
\]
for suitable $b_m\in \mathbb Z_2$ that satisfy the following conditions:

\begin{enumerate}
\item 
$b_0 \equiv 1 \pmod2$; 
\item $b_0+b_1 \equiv 3 \pmod4$;
\item $|b_m|_2=1$, $m\geq 2$;
\item $b_2+b_3\equiv2\pmod4$;
\item $\sum_{m=2^{n-1}}^{2^{n}-1} b_m \equiv 0\pmod4$, $n\geq 3$.
\end{enumerate}
\end{theo}
\begin{remark}
\label{rem:eq-er-sph}
It is an elementary exercise to show that condition 2 in the statement of Theorem \ref{thm:ergnew}
  can be replaced by the condition $b_1-b_0\equiv 1\pmod 4$.
\end{remark}

\section{Ergodicity of 1-Lipschitz dynamical systems on 2-adic spheres}
In this section we prove ergodicity criterion for 1-Lipschitz dynamics on
2-adic spheres, so further  $p=2$ and  $f\colon\mathbb Z_2\rightarrow\mathbb Z_2$  is a  1-Lipschitz function. Let $\mathbf S_{2^{-r}}(a)$ be a sphere of radius $2^{-r}$ with a center at the point $a\in \left\{0,\ldots,2^r-1\right\}$, and
let the sphere $\mathbf S_{2^{-r}}(a)$ be invariant under action of  $f$; that is,
let $f(\mathbf S_{2^{-r}}(a))\subset \mathbf S_{2^{-r}}(a)$.
As $p=2$, the sphere $\mathbf S_{2^{-r}}(a)$ coincides with the ball $\mathbf B_{2^{-r-1}}(a+2^r)$ of radius $2^{-r-1}$ centered at the point $a+2^r$: $\mathbf S_{2^{-r}}(a)=\left\{a+2^r+2^{r+1}x\colon x\in \mathbb Z_2\right\}=\mathbf B_{2^{-r-1}}(a+2^r)$.
Therefore the sphere $\mathbf S_{2^{-r}}(a)$ is $f$-invariant if and only
if $f(a+2^r+2^{r+1}\mathbb Z_p)\subset a+2^r+2^{r+1}\mathbb Z_p$; that is, if and only if
\begin{equation}
\label{eq:f-invar}
f(a+2^r)\equiv a+2^r\pmod{2^{r+1}}
\end{equation}
as a 1-Lipschitz function maps a ball of radius $2^{-\ell}$ into a ball of
radius $2^{-\ell}$.

Further, as $f$ is 1-Lipschitz, we can represent $f\colon\mathbb Z_2\rightarrow\mathbb Z_2$ as
\begin{equation}
\label{eq:f-rep}
f(a+2^r+2^{r+1}x)=f(a+2^r)
+2^{r+1}g(x);
\end{equation}
then $g\colon\mathbb Z_2\rightarrow\mathbb Z_2$ is a 1-Lipschitz function. The following proposition
holds:
\begin{prop}
\label{prop:ball-space}
The function $f$ is ergodic on the sphere $\mathbf S_{2^{-r}}(a)$ if and only if $f(a+2^r)\equiv a+2^r\pmod{2^{r+1}}$ and the function $g$ defined as 
\begin{equation}
\label{eq:g-rep}
g(x)=\frac{f(a+2^r+2^{r+1}x)-
(a+2^r)
}{2^{r+1}}
\end{equation}
is a 1-Lipschitz ergodic transformation on $\mathbb Z_2$.
\end{prop}
\begin{proof} 
The first of conditions of the proposition just means that the
sphere $\mathbf S_{2^{-r}}(a)$ is $f$-invariant, see \eqref{eq:f-invar}. In view of that condition,
Condition 2 of Lemma \ref{lem:transit} holds then if and only if
the function $g$ is transitive modulo $2^t$ for all $t=1,2,3,\ldots$
However, the latter condition is equivalent to the ergodicity of the function $g$ on $\mathbb Z_2$ by Theorem \ref{thm:erg-main}.
\end{proof}
Now given a 1-Lipschitz function $f\colon\mathbb Z_2\rightarrow\mathbb Z_2$, in view of Theorem \ref{thm:comp}
and \eqref{eq:vdp-comp},
$f$ has a unique representation via van der Put series:
\begin{equation}
\label{eq:f-vdp}
f(x)=\sum_{m=0}^\infty B_f(m) \chi (m,x)=\sum_{m=0}^\infty 2^{\left\lfloor \log_2m \right\rfloor}b_f(m) \chi (m,x),
\end{equation}
where $b_f(m)\in\mathbb Z_2$; so $B_f(m)=2^{\left\lfloor\log_2m \right\rfloor}b_f(m)$
for all $m=0,1,2,\ldots$.
\begin{theo}
\label{thm:ergsph}
The function $f$ represented by van der Put series \eqref{eq:f-vdp} is ergodic on the sphere $\mathbf S_{2^{-r}}(a)$ if and only if the following conditions  hold simultaneously:
\begin{enumerate}
\item  $f(a+2^r)\equiv a+2^r+2^{r+1}\pmod{2^{r+2}}$;
\item $\left| b_f(a+2^r+m\cdot 2^{r+1}) \right|_2 =1$,
for $m\geq 1$;   
\item $b_f(a+2^r+2^{r+1})
\equiv 1 \pmod 4$;
\item $b_f(a+2^r+2^{r+2})+b_f(a+2^r+3\cdot 2^{r+1})\equiv 2 \pmod 4$;
\item $\sum_{m=2^{n-1}}^{2^n-1} 
b_f(a+2^r+m\cdot 2^{r+1})\equiv 0 \pmod 4,\; \text{for}\; n\geq 3$.
\end{enumerate}
\end{theo}
\begin{proof}
Represent $f$ 
by \eqref{eq:f-rep} and $g$ by \eqref{eq:g-rep},
then calculate van der Put coefficients $B_m$ of the function 
$g(x)=\sum_{m=0}^\infty B_m \chi (m,x)$ as follows.  Given $m\in\{2,3,\ldots\}$, denote $\acute
m=m-2^{\lfloor\log_2m\rfloor}$; then using formula \eqref{eq:vdp-coef} for
van der Put coefficients $B_m$ and $B_f(m)$ we find that 
\begin{align}
B_0=& g(0)= \frac{f(a+2^r)-
(a+2^r)
}{2^{r+1}};\label{eq:g-f-0} \\
B_1=& g(1)= \frac{f(a+2^r+2^{r+1})-
(a+2^r)
}{2^{r+1}};\label{eq:g-f-1} \\
B_m=B_{\acute{m}+2^n}=& g(\acute{m}+2^n)-g(\acute{m}) =\notag \\
=& \frac{f(a+2^r+2^{r+1}\acute{m}+2^{n+r+1})-
(a+2^r)
}{2^{r+1}} -\notag \\
& 
\frac{f(a+2^r+2^{r+1}\acute{m})-
(a+2^r)
}{2^{r+1}}=\notag \\
& \frac{f(a+2^r+2^{r+1}\acute{m}+2^{n+r+1})-f(a+2^r+2^{r+1}\acute{m})}{2^{r+1}} =\notag\\
& \frac{B_f(a+2^r+2^{r+1}\acute{m}+2^{n+r+1})}{2^{r+1}},\label{eq:g-f-m} \;
(m\geq2, n=\lfloor\log_2m\rfloor).
\end{align} 
As $g$ is 1-Lipschitz, Theorem  \ref{thm:comp} yields that for every $k=0,1,2,\ldots$ there exists $b_k\in\mathbb Z_p$ such that  $B_k=b_k2^{\lfloor\log_2k\rfloor}$.
So from \eqref{eq:g-f-m} it follows that if $m\ge 2$ and $n=\lfloor\log_2m\rfloor$
then
\begin{equation}
\label{eq:g-f-m-fin}
b_m=\frac{B_m}{2^n}=\frac{B_f(a+2^r+2^{r+1}\acute{m}+2^{n+r+1})}{2^{n+r+1}}=b_f{(a+2^r+ m\cdot2^{r+1})}.
\end{equation}
Now apply Theorem \ref{thm:ergnew}.  In view of \eqref{eq:g-f-m-fin}, condition  3 (for $m\ge
2$) of  Theorem \ref{thm:ergnew} 
is equivalent to condition 2 of the  theorem under proof.   By the same reason, conditions 4 and 5  of Theorem \ref{thm:ergnew} 
accordingly are equivalent to conditions 4 and 5 of the  theorem under  proof. Combining  \eqref{eq:g-rep} with \eqref{eq:vdp-coef}, we see that condition 1 of Theorem \ref{thm:ergnew}
is equivalent to condition 1 of the  theorem under proof.

Finally, combining \eqref{eq:g-f-0} and \eqref{eq:g-f-1} with \eqref{eq:vdp-coef}
we see that
\begin{multline}
\label{eq:b1-b0}
b_1-b_0=g(1)-g(0)=\frac{f(a+2^r+2^{r+1})-f(a+2^r)}{2^{r+1}}=\\
\frac{B_f(a+2^r+2^{r+1})}{2^{r+1}}=b_f(a+2^r+2^{r+1})
\end{multline}
By Remark \ref{rem:eq-er-sph} to Theorem \ref{thm:ergnew},  
\eqref{eq:b1-b0} proves condition 3 (as well
as condition 1 for $m=1$) of the theorem under proof. 
\end{proof}
\section{Ergodicity  of perturbed  monomial systems on 2-adic spheres around
1}
In this section, we study ergodicity of a function  of the form $f(x)=x^s+2^{r+1}u(x)$ on the sphere $\mathbf S_{2^{-r}}(1)=\left\{1+2^r+2^{r+1}x\colon x\in \mathbb Z_2\right\}$, $r\ge
1$, $s\in\mathbb N$. The
perturbation function $u$ is assumed to be 1-Lipschitz.

\begin{theo}
\label{thm:pert}
Let $u\colon \mathbb Z_2\rightarrow \mathbb Z_2$ be an arbitrary 1-Lipschitz function, let
$s,r\in\mathbb N$.
The function $f(x)=x^s+2^{r+1}u(x)$ is ergodic on the sphere $\mathbf S_{2^{-r}}(1)=\left\{1+2^r+2^{r+1}x\colon x\in \mathbb Z_2\right\}$ 
 if and only if $s\equiv 1 \pmod 4$ and $u(1)\equiv 1 \pmod 2$.
\end{theo}
\begin{proof}
As $u$ is 1-Lipschitz, $u(1+2^r+2^{r+1}x)=u(1+2^r)+2^{r+1}\xi(x)$ for a suitable
map $\xi\colon\mathbb Z_2\rightarrow\mathbb Z_2$; it is an exercise to prove that $\xi$ is also 1-Lipschitz. Combining \eqref{eq:g-rep}, \eqref{eq:f-rep} and Newton's binomial we get
\begin{multline}
\label{eq:pert-rep}
g(x)=\frac{f(1+2^r+2^{r+1}x)-
(1+2^r)
}{2^{r+1}}=\\ 
u(1+2^r+2^{r+1}x)+ \frac{1}{2^{r+1}}\left( (1+2^r+2^{r+1}x)^s-(1+2^r)\right)=\\
2^{r+1}\xi(x)+u(1+2^r)+xs+2^{r-1}(1+2x)^2\binom{s}{2}+2^{2r-1}(1+2x)^3\binom{s}{3}+\cdots
\end{multline}
By Proposition \ref{prop:ball-space} the function $f$ is ergodic on the sphere
$\mathbf S_{2^{-r}}(1)$ if and only  the function $g$ is ergodic on $\mathbb Z_p$.
As $r\ge 1$ and $\xi$ is 1-Lipschitz, the right-hand side of \eqref{eq:pert-rep}
is ergodic on $\mathbb Z_p$ if and only if the polynomial
$$
v(x)=u(1+2^r)+xs+2^{r-1}(1+2x)^2\binom{s}{2}+2^{2r-1}(1+2x)^3\binom{s}{3}+\cdots
$$
in variable $x$ is ergodic on $\mathbb Z_p$: This follows from \cite[Proposition
9.29]{ANKH} where it is shown in particular that given  1-Lipschitz functions
$t,w\colon\mathbb Z_2\rightarrow\mathbb Z_2$, the function $t+4w$ is ergodic on $\mathbb Z_2$ if and
only if the function $t$ is ergodic on $\mathbb Z_2$. Nonetheless, we would
better deduce the same claim 
from Theorem \ref{thm:ergsph} to obtain a formula that will be used later
in the current proof.

We first note that conditions 2-5 of Theorem \ref{thm:ergsph}
are all ``modulo 4'', meaning the conditions do not depend on terms of order greater than
1 of canonic 2-adic representations of van der Put coefficients of the function.
Then, condition 1 of Theorem  \ref{thm:ergsph} holds if and only if the following congruence holds:
$$
f(1+2^r)=(1+2^r)^s+2^{r+1}u(1+2^r)\equiv 1+2^r+2^{r+1}\pmod{2^{r+2}}.
$$
As $u$ is 1-Lipschitz, the latter congruence is equivalent to the congruence
$1+2^rs+2^{2r}\binom{s}{2}+2^{r+1}u(1)\equiv 1+2^r+2^{r+1}\pmod{2^{r+2}}$
which by Newton's binomial is equivalent to the congruence
\begin{equation}
\label{eq:pert-1}
s+2^r\binom{s}{2}+2u(1)\equiv 3\pmod 4
\end{equation}
(and therefore $s\equiv 1\pmod 2$: $s=2\tilde s+1$ for suitable $\tilde s\in\mathbb N_0$).
So we conclude that the summand $2^{r+1}\xi(x)$ has no effect on ergodicity
of the function $g$ on $\mathbb Z_2$: The latter function is ergodic if and only if
the polynomial $v$ is ergodic on $\mathbb Z_2$.

Further, according to \cite{Lar} (see also \cite[Corollary 4.71]{ANKH}) a polynomial over $\mathbb Z_2$ is ergodic on $\mathbb Z_2$ if and only if it is transitive modulo 8. This
already proves the theorem if $r\ge 3$ since in the latter case 
$v(x)\equiv u(1+2^r)+xs\pmod 8$; however, in force of a well know criterion of ergodicity
of affine maps (see e.g. \cite[Theorem 4.36]{ANKH}), the affine map $x\mapsto u(1+2^r)+xs$
is ergodic on $\mathbb Z_2$ if and only if $u(1+2^r)\equiv 1\pmod 2$ (so $u(1)\equiv
1\pmod 2$ as $u$ is 1-Lipschitz) and $s\equiv
1\pmod 4$.

To complete the proof of the theorem, only  cases $r=1$ and $r=2$ are to be considered.
If $r=2$ then $v(x)\equiv u(1+2^2)+x^s+2\binom{s}{2}+4\binom{s}{2}x\pmod
8$ and we conclude the proof as in the case $r\ge 2$. In the remaining case
$r=1$ we have that
\begin{multline}
\label{eq:pol-2}
v(x)\equiv u(1+2)+xs+\binom{s}{2}+2\binom{s}{2}x+4\binom{s}{2}x^2+\\
2\binom{s}{3}+6\binom{s}{3}x+4\binom{s}{4}\pmod 8.
\end{multline}
However, if $r=1$ then congruence \eqref{eq:pert-1} implies the congruence
$1+2\tilde s+2(2\tilde s+1)\tilde s+2u(1)\equiv 3\pmod 4$ (remind that $s=2\tilde s+1$)
which is equivalent to the congruence $u(1)\equiv 1\pmod 2$. Yet the latter
congruence implies that necessarily $\tilde s\equiv 0 \pmod 2$: If otherwise, then from \eqref{eq:pol-2} it follows that $v(x)\equiv u(1)+xs+1\equiv x\pmod
2$ which means that $v(x)$ is not transitive modulo 2 and thus the polynomial
$v(x)$ can not be ergodic on $\mathbb Z_2$ (cf. Theorem \ref{thm:erg-main}).  On
the other hand, if $u(1)\equiv 1\pmod 2$ and $s\equiv 1\pmod 4$ (i.e., $s=4\hat
s+1$ for some $\hat s\in\mathbb N_0$) then the
polynomial in the right-hand side of \eqref{eq:pol-2} is congruent modulo
8 to the polynomial $u(3)+x+6\hat s$ which induces an ergodic affine transformation
on $\mathbb Z_2$ (cf. \cite[Theorem 4.36]{ANKH}) since $u(3)\equiv u(1)\equiv 1\pmod 2$ (recall
that $u$ is 1-Lipschitz). Thus, the polynomial $v(x)$ (and whence the function
$g$) are also ergodic on $\mathbb Z_2$. This completes the proof.
\end{proof}
\section*{Acknowledgement}
The first of authors would like to
express his gratitude  for hospitality
during his stay in Linnaeus University.


\begin{thebibliography}{99}

\bibitem{Albeverio1}  
S. Albeverio, A. Khrennikov, and P. E. Kloeden. Memory retrieval as a p-adic dynamical system.
\emph{BioSystems}. {\bf 49} (1999), 105--115.

\bibitem{Albeverio2} 
S. Albeverio, A. Khrennikov, B. Tirozzi, and D. De Smedt. p-adic dynamical systems. \emph{Theor. Math. Phys.} {\bf 114} (1998), 276--287.

\bibitem{ANKH} 
V.~Anashin and A.~Khrennikov.
 \emph{Applied Algebraic Dynamics. (de Gruyter
  Expositions in Mathematics, vol.~49)}. Walter de Gruyter, Berlin--New York,
  2009.


\bibitem{Unif0} 
V. Anashin. Uniformly distributed sequences of p-adic integers. \emph{Math. Notes} {\bf 55} (1994), 109--133.

\bibitem{Unif} 
V.~Anashin.
Uniformly distributed sequences of $p$-adic integers. 
\emph{Discrete Math. Appl.} {\bf 12}(6) (2009), 527--590.

\bibitem{Erg} 
V.~Anashin.
Ergodic transformations in the space of $p$-adic integers. 
\emph{$p$-adic Mathematical Physics. 2-nd Int'l Conference (Belgrade,
  Serbia and Montenegro 15--21 September 2005) (AIP
  Conference Proceedings, 826)},  American
  Institute of Physics, Melville, New York, 2006, pp. 3--24.

\bibitem{me-NATO}
V.~Anashin.
Non-{A}rchimedean theory of {T}-functions.
\emph{Proc. Advanced Study Institute Boolean Functions in
  Cryptology and Information Security (NATO Sci. Peace
  Secur. Ser. D Inf. Commun. Secur., 18)}, IOS Press, Amsterdam, 2008, pp. 33--57.


\bibitem{DAN} 
V.~S.~Anashin, A.~Yu.~Khrennikov and E.~I.~Yurova.  
Characterization of ergodicity of $p$-adic dynamical systems
by using the van der Put basis. \emph{Doklady Akademii Nauk}, {\bf 438}(2) (2011),
151--153.

\bibitem{AKY}
V. Anashin, A. Khrennikov and E. Yurova. Using van der Put basis to determine if a 2-adic function is measure-preserving or ergodic w.r.t. Haar measure. \emph{Advances in Non-Archimedean Analysis (Contemporary Mathematics, 551)}, American Mathematical Society, Providence,
RI,
 2011, pp. 35--38.

\bibitem{Arrowsmith1}  
D. K. Arrowsmith and F. Vivaldi. Some p-adic representations of the Smale horseshoe. \emph{Phys.
Lett. A}, {\bf 176} (1993), 292--294.

\bibitem{Arrowsmith2} 
D. K. Arrowsmith and F. Vivaldi. Geometry of p-adic Siegel discs. \emph{Physica D}, {\bf 71} (1994), 222--236.


\bibitem{Benedetto1}  
R. Benedetto. p-adic dynamics and Sullivan's no wandering domain theorem. \emph{Compos. Math.} {\bf 122} (2000), 281--298.

\bibitem{Benedetto2} 
R. Benedetto. Hyperbolic maps in p-adic dynamics. \emph{Ergod. Theory and Dyn. Sys.} {\bf 21} (2001), 1--11.

\bibitem{Benedetto3} 
R. Benedetto. Components and periodic points in non-Archimedean dynamics. \emph{Proc. London
Math. Soc.}, {\bf 84} (2002), 231--256.

\bibitem{Benedetto4} 
R. Benedetto. Heights and preperiodic points of polynomials over function fields. \emph{Int. Math.
Res. Notices}, {\bf 62} (2005), 3855--3866.

\bibitem{Chabert} 
J.-L. Chabert, A.-H. Fan, Y. Fares,   Minimal dynamical systems on a discrete valuation domain. \emph{Discrete and Continuous Dynamical Systems - Series A}, {\bf 25} (2009), 777--795.
 
\bibitem{Parry}  
Z. Coelho and W. Parry. Ergodicity of p-adic multiplication and the distribution of Fibonacci
numbers. \emph{Topology, ergodic theory, real algebraic geometry (Amer. Math. Soc.
Transl. Ser., 202)} American Mathematical Society, Providence, RI, 2001, pp. 51--70.

\bibitem{Fan1}  
A.-H. Fan, M.-T. Li, J.-Y. Yao, and D. Zhou. p-adic affine dynamical systems and applications.
\emph{C. R. Acad. Sci. Paris Ser. I}, {\bf 342} (2006), 129--134.

\bibitem{Fan2} 
A.-H. Fan, M.-T. Li, J.-Y. Yao, and D. Zhou. Strict ergodicity of affine p-adic dynamical systems.
\emph{Adv. Math.} {\bf 214} (2007), 666--700.

\bibitem{Fan3} 
A.-H. Fan, L. Liao, Y.-F. Wang, and D. Zhou. p-adic repellers in $\mathbb Q_p$ are subshifts of finite type. \emph{C. R. Math. Acad. Sci. Paris.}
{\bf 344} (2007), 219--224. 

\bibitem{Rivera1} 
C. Favre and J. Rivera-Letelier. Th´eor`eme d'equidistribution de Brolin en dynamique p-adique. \emph{C. R. Math. Acad. Sci. Paris}, {\bf 339} (2004),
271--276.


\bibitem{Gundlach1} 
M. Gundlach, A. Khrennikov, K.-O. Lindahl. On ergodic behaviour of
$p$-adic dynamical systems. \emph{Infinite Dimensional Analysis,
Quantum Prob. and Related Fields,} {\bf 4} (2001), 569--577.

\bibitem{Gundlach2} 
M. Gundlach, A. Khrennikov, K.-O. Lindahl. Topological transitivity for p-adic
dynamical systems. 
\emph{p-adic functional analysis (Lecture notes in pure
and apllied mathematics, 222)}, Dekker, New York,
2001, pp. 127--132.


\bibitem{Khrennikov and Nilsson} 
A. Khrennikov and M. Nilsson. \emph{p-adic deterministic and random dynamics}. Kluwer, Dordrecht, 2004

\bibitem{Lar}
M.~V.~Larin, Transitive polynomial transformations of residue class rings, \emph{Discrete
Math. Appl.} {\bf 12} (2002), 141--154. 

\bibitem{LinTaoYang} 
D.-D. Lin, T. Shi, and Z.-F. Yang.
Ergodic theory over $\mathbb F_2[[X]]$.
\emph{Finite Fields Appl.} {\bf 18} (2012), 473--491.

\bibitem{Lindhal} 
K-O. Lindhal. On Siegel's linearization theorem for fields of prime characteristic. \emph{Nonlinearity},
{\bf 17} (2004), 745--763.

\bibitem{Mahler}
K.~Mahler.
\emph{$p$-adic numbers and their functions}, Cambridge Univ. Press, Cambridge,
1981.

\bibitem{vdp}
M. van der Put,
\emph{Alg\`ebres de fonctions continues $p$-adiques}, Universiteit Utrecht,
1967.

\bibitem{Rivera2} 
J. Rivera-Letelier. \emph{Dynamique des fonctions rationelles sur des corps locaux. (PhD thesis)}, Orsay,
2000.

\bibitem{Rivera3} 
J. Rivera-Letelier. Dynamique des fonctions rationelles sur des corps locaux. \emph{Asterisque},
{\bf 147} (2003), 147--230.

\bibitem{Rivera4} 
J. Rivera-Letelier. Expace hyperbolique p-adique et dynamique des fonctions rationelles. \emph{Compos.
Math.}, {\bf 138} (2003), 199--231.

\bibitem{De Smedt1} 
A. De Smedt and A. Khrennikov. A p-adic behaviour of dynamical systems. \emph{Rev. Mat. Complut.},
{\bf 12} (1999), 301--323.

\bibitem{Shihov}
W.~H.~Schikhof
\emph{Ultrametric calculus. An introduction to $p$-adic analysis}.
Cambridge University Press, Cambridge, 1984

\bibitem{Silverman} 
J. H. Silverman. \emph{The arithmetic of dynamical systems  (Graduate Texts in Mathematics, 241,)} Springer, New York, 2007.


\bibitem{Vivaldi1}  
F. Vivaldi. The arithmetic of discretized rotations. 
\emph{$p$-adic Mathematical Physics. 2-nd Int'l Conference (Belgrade,
  Serbia and Montenegro 15--21 September 2005) (AIP
  Conference Proceedings, 826)},  American
  Institute of Physics, Melville, New York, 2006, pp. 162--173.

\bibitem{Vivaldi2} 
F. Vivaldi and I. Vladimirov. Pseudo-randomness of round-off errors in discretized linear maps
on the plane. \emph{Int. J. of Bifurcations and Chaos}, {\bf 13} (2003),
3373--3393.

\bibitem{Vivaldi} 
F. Vivaldi, Algebraic and arithmetic dynamics bibliographical database,
\url{http://www.maths.qmul.ac.uk/~fv/database/algdyn.pdf}

\bibitem{Ya} 
E.~I.~Yurova.
Van der Put basis and $p$-adic dynamics 
\emph{$p$-Adic Numbers, Ultrametric Analysis and Applications}, {\bf 2}(2)(2010)
175--178.

\end{thebibliography}
\end{document}